\documentclass[12pt]{amsart}
\usepackage{amssymb,amsmath,amsfonts,latexsym,mathrsfs}
\usepackage{bm}
\usepackage[all,cmtip]{xy}
\usepackage{amscd}

\usepackage[shortlabels]{enumitem}
\usepackage[colorlinks=true,
linkcolor=blue,
citecolor=blue,
urlcolor=blue]{hyperref}

\newcommand{\bea}{\begin{eqnarray}}
	\newcommand{\eea}{\end{eqnarray}}

\newcommand{\cla}{\mathcal{A}}
\newcommand{\clb}{\mathcal{B}}
\newcommand{\clc}{\mathcal{C}}
\newcommand{\cld}{\mathcal{D}}
\newcommand{\cle}{\mathcal{E}}

\newcommand{\clh}{\mathcal{H}}
\newcommand{\cli}{\mathcal{I}}
\newcommand{\clk}{\mathcal{K}}
\newcommand{\cll}{\mathcal{L}}
\newcommand{\clm}{\mathcal{M}}
\newcommand{\cln}{\mathcal{N}}

\newcommand{\clr}{\mathcal{R}}
\newcommand{\cls}{\mathcal{S}}
\newcommand{\clt}{\mathcal{T}}

\newcommand{\z}{\bm{z}}

\newcommand{\K}{\bm{k}}
\newcommand{\AL}{\bm{l}}
\newcommand{\M}{\bm{m}}

\newcommand{\T}{\mathbb{T}}
\newcommand{\D}{\mathbb{D}}
\newcommand{\C}{\mathbb{C}}
\newcommand{\Z}{\mathbb{Z}}

\def\textmatrix#1&#2\\#3&#4\\{\bigl({#1 \atop #3}\ {#2 \atop #4}\bigr)}
\def\dispmatrix#1&#2\\#3&#4\\{\left({#1 \atop #3}\ {#2 \atop #4}\right)}
\newcommand{\be}{\begin{equation}}
	\newcommand{\ee}{\end{equation}}
\newcommand{\ben}{\begin{eqnarray*}}
	\newcommand{\een}{\end{eqnarray*}}

\newcommand{\NI}{\noindent}

\newcommand{\bi}{\begin{itemize}}
	\newcommand{\ei}{\end{itemize}}

\newcommand{\wh}{\widehat}

\newtheorem*{theoremA}{Theorem A}
\newtheorem*{theoremB}{Theorem B}
\newtheorem*{theoremC}{Theorem C}

\theoremstyle{definition}

\theoremstyle{plain}

\newtheorem{thm}{Theorem}[section]
\newtheorem{cor}[thm]{Corollary}
\newtheorem{lem}[thm]{Lemma}
\newtheorem{prop}[thm]{Proposition}
\theoremstyle{definition}
\newtheorem{defn}[thm]{Definition}
\newtheorem{rem}[thm]{Remark}

\numberwithin{equation}{section}

\let\phi=\varphi

\begin{document}

\begin{abstract}
In this paper, we obtain a complete classification of Toeplitz + Hankel-type operators on the vector-valued Hardy space $H^2_{\mathcal{E}}(\mathbb{D}^n)$ over the polydisc 
$\mathbb{D}^n$ in $\mathbb{C}^n$ for $n\geq 1$. We also characterize the paired operators on $L^2(\T^n)$. Furthermore, we give a complete characterization for the class of essentially Toeplitz + essentially Hankel operators on the vector-valued Hardy space $H^2_{\cle}(\D)$ for finite-dimensional Hilbert space $\cle$.
\end{abstract}
	

\title[Generalization of Toeplitz + Hankel operators]{Generalization of Toeplitz + Hankel operators in the polydisc}

	\author[Babbar]{Kritika Babbar}
	\address{Indian Institute of Technology Roorkee, Department of Mathematics,
		Roorkee-247 667, Uttarakhand, India}
	\email{kritikababbariitr@gmail.com, kritika@ma.iitr.ac.in}
	
	\author[Javed]{Mo Javed}
	\address{Indian Institute of Technology Roorkee, Department of Mathematics,
		Roorkee-247 667, Uttarakhand, India}
	\email{javediitr07@gmail.com, mo\_j@ma.iitr.ac.in}
	
	\author[Maji]{Amit Maji}
	\address{Indian Institute of Technology Roorkee, Department of Mathematics,
		Roorkee-247 667, Uttarakhand, India}
	\email{amit.maji@ma.iitr.ac.in, amit.iitm07@gmail.com ({Corresponding author)}}
	
\maketitle

\section{Introduction}\label{Sec: Introduction}
The study of Toeplitz and Hankel operators has become a major area of interest for several mathematicians from the last few decades due to their rich structural properties and numerous applications in function theory, control theory, quantum physics, signal processing, time series analysis, and so on. A systematic study of Toeplitz operators was initiated after the seminal paper of Brown and Halmos \cite{Brown}, where they obtained the algebraic characterization of Toeplitz operators on the Hardy space $H^2(\D)$. On the other hand, the celebrated Nehari's theorem \cite{Nehari} characterized the bounded Hankel operators in terms of bounded symbols. Peller's monograph \cite{Peller} is a valuable source in the development of the theory of Hankel operators and their applications. Inspired by these developments, several mathematicians extended the theory of Toeplitz and Hankel operators in the polydisc setting. Recently, Maji, Sarkar, and Sarkar \cite{Maji} extended the Brown-Halmos type characterization in the polydisc setting. Cotlar and Sadosky \cite{Cotlar} initiated the systematic study of Hankel operators in the polydisc. Gu \cite{Gu} established the structural properties of Hankel operators on $H^2(\D^n)$ for $n \geq 1$ and their relation with the Toeplitz operators. For further results regarding Toeplitz and Hankel operators on the polydisc, one may refer to \cite{Ding, Gu1}.

The class of Toeplitz + Hankel operators has also emerged as a distinguishable class of operators that has been studied for many years. The invertibility and Fredholmness of Toeplitz + Hankel operators are studied extensively by many mathematicians \cite{Basor, Deift, Didenko, EhrhardtJFA}. In addition, their structural and algebraic properties on scalar as well as vector-valued Hardy spaces have been explored in several works (see \cite{Ehrhardt, GuJMAA, Sang} and the references within). Recently, Das, Das, and Sarkar \cite{N-DAS-S-DAS-SARKAR} characterized the Toeplitz + Hankel operators on the $\cle$-valued Hardy space $H^2_\cle(\D)$.

In this article, we study the corresponding operator class on the polydisc. More precisely, motivated by the classical intertwining relation for Hankel operators, we consider bounded operators $T$ on $H^2_{\cle}(\D^n)$ satisfying
\[
T_{z_i}^*T = TT_{z_i}, \quad \text{for all } i=1,\dots,n
\]
and refer to such operators as {\it Hankel-type} operators. We emphasize that this terminology is adopted as a definition for the purposes of this paper and does not assert that the above intertwining relations characterize the classical Hankel operators on the polydisc. We then provide a complete classification of Toeplitz + Hankel-type operators on $H^2_{\cle}(\D^n)$ as follows:
	
\begin{theoremA}
Let $\cle$ be a Hilbert space and $A\in \clb(H^2_{\cle}(\D^n))$. Then the following are equivalent:
\begin{enumerate}
\item $A=$ Toeplitz + Hankel-type.
\item For all $j=1,2,\dots,n$, $T_{z_j}^*AT_{z_j}-A$ is a Hankel-type operator.
\item For all $j=1,2,\dots,n$, $T_{z_j}^*A-AT_{z_j}$ is a Toeplitz operator.
\item For all $j,k=1,2,\dots,n$, we have
			\[
			AT_{z_j}+T_{z_j}^*T_{z_k}^*AT_{z_k}=T_{z_j}^*A+T_{z_k}^*AT_{z_j}T_{z_k}.
			\]
\end{enumerate}
\end{theoremA}

Another important aspect of this paper is the study of the class of paired operators.  For a Hilbert space $\cll$, let $P$ and $Q$ be orthogonal projections on $\cll$ such that $P + Q=I_\cll$. Let $X, Y \in \clb(\cll)$. The \textit{paired operator} $S_{X,Y}$ on $\cll$ is defined by 
	\[
	S_{X,Y}=XP+ YQ.
	\]
Denoting $P_+=P_{H^2(\D^n)}$, the orthogonal projection of $L^2(\T^n)$ onto $H^2(\D^n)$ and $P_-=I_{L^2(\T^n)}-P_+$. Clearly, $P_{+}$ and $P_{-}$ are \emph{complementary projections}, i.e., $P_{+}+P_{-}=I_{L^2(\T^n)}$. We define the paired operator on $L^2(\T^n)$ as follows:
	
\begin{defn}
For $\phi,\psi \in L^\infty(\T^n)$, the paired operator $S_{\phi,\psi} : L^2(\T^n)\rightarrow L^2(\T^n)$ is defined as
		\[
		S_{\phi,\psi}=M_\phi P_+ +M_\psi P_-,
		\]
where $M_{\phi}$ and $M_{\psi}$ are the Laurent operators on $L^2(\T^n)$.
\end{defn}

These paired operators are dilations of Toeplitz operators and are also referred to as \textit{singular integral operators} by several authors. The structural and algebraic properties of $S_{\phi,\psi}$ have been studied extensively over the years (see \cite{CGu, Mikhlin}). In \cite{N-DAS-S-DAS-SARKAR}, an algebraic characterization of paired operators on $L^2(\T)$ is obtained. So our next goal is to give a generalized version of this characterization on $L^2(\T^n)$. We obtain a characterization result for paired operators on $L^2(\T^n)$ as follows:
\begin{theoremB}
		Let $A\in\clb(L^2(\T^n))$. Then the following are equivalent:
		\begin{enumerate}
		\item $A$ is a paired operator.
		\item $A=M_{z_j}^*AM_{z_j}P_{+}+M_{z_j}AM_{z_j}^*P_{-}$ for all $j=1,\dots,n$.
\end{enumerate}
\end{theoremB}
	
Our next aim is to generalize the notion of Toeplitz + Hankel operators and to characterize the class of essentially Toeplitz + essentially Hankel operators on $H^2_{\cle}(\D)$, where $\cle$ is a finite-dimensional Hilbert space, in particular, for scalar-valued Hardy space $H^2(\D)$. We denote $\clt_{ess}$ and $\clh_{ess}$ are the class of essentially Toeplitz and essentially Hankel operators, respectively. 
We give an intrinsic characterization as follows:
	
\begin{theoremC}
Let $\cle$ be a Hilbert space and $A \in \clb(H^2_{\cle}(\D))$. Then the following are equivalent:
\begin{enumerate}
			\item $A \in \clt_{ess}+ \clh_{ess}$.
			\item $T_z^*AT_z-A$ is an essentially Hankel operator.
			\item $T_z^*A-AT_z$ is an essentially Toeplitz operator.
			\item $AT_z+T_z^{*2}AT_z-T_z^*A-T_z^*AT_z^2$ is compact.
\end{enumerate}
\end{theoremC}

The remainder of the article is outlined as follows: In Section \ref{Sec: Preliminaries}, we introduce the notation and preliminary results concerning Hardy spaces over the polydisc and the operator class defined by the intertwining relations: $T_{z_i}^*T=TT_{z_i}$ for all $i=1,\dots,n.$ We establish a representation of such operators in terms of operator-valued strongly $H^2$-functions. Section \ref{Sec: Toeplitz + Hankel} is devoted to classifying the Toeplitz + Hankel-type operators on the $\cle$-valued Hardy space over the polydisc $\D^n$. The characterization results for paired operators on $L^2(\T^n)$ have been presented in Section \ref{Sec: Paired operators}. In Section \ref{Sec:ET+EH}, essentially Toeplitz + essentially Hankel operators are completely classified on $H^2_{\cle}(\D)$ for finite dimensional Hilbert space $\cle$.

\section{Preliminaries}\label{Sec: Preliminaries}
	
In what follows, $n \geq 1$ is a positive integer, and all Hilbert spaces considered in this paper are complex and separable. Let $\D=\{z\in\C : |z|<1\}$ be the open unit disc in the complex plane $\C$ and the unit circle $\T=\{z\in\C : |z|=1\}$ be the boundary of $\D$. We denote by $\D^n$ the open unit polydisc in $\C^n$ and $\T^n$ the distinguished boundary of $\D^n$. For a Hilbert space $\cle$, denote by $L^2_\cle(\T^n)$, the space of all Lebesgue measurable functions $f$ on $\T^n$ such that 
	\[
	\|f\|^2=\int_{\T^n} \left\|f(\zeta_1,\dots,\zeta_n)\right\|^2_{\cle}\, d\mu <\infty,
	\]
where $d\mu$ is the normalized Lebegue measure on $\T^n$. The \textit{$\cle$-valued Hardy space} $H^2_\cle(\D^n)$ over the polydisc $\D^n$ is the Hilbert space of all $\cle$-valued holomorphic functions $f$ on $\D^n$ such that
	\[
	\|f\|^2=\sup_{0\leq r<1}\left(\int_{\T^n}\left\|f(r\zeta_1, \dots, r\zeta_n)\right\|^2_\cle \, d\mu\right) <\infty.
	\]
It is well-known that $H^2_{\cle}(\D^n)$ is isometrically isomorphic to the closed subspace of $L^2_\cle(\T^n)$ which consists of all such functions whose negative Fourier coefficients are zero and this characterization can be proved by means of radial limits, see \cite{sznagy} for complete details. Let $(T_{z_1}, \ldots, T_{z_n})$ be the $n$-tuple of shift operators, where
\begin{equation}\label{Eq: Definition of shift T-zi}
		T_{z_i}f=z_i f, \quad (f \in H^2_{\cle}(\D^n)),
\end{equation}
for each $i\in\{1,2, \ldots, n\}$. For $\z=(z_1,\dots,z_n)\in \T^n$ and the multi-index $\K=(k_1,\dots,k_n)\in\Z^n$, we write $\z^{\K}=z_1^{k_1}\cdots z_n^{k_n}$. Furthermore, we denote $L^\infty_{\clb(\cle)}(\T^n)$ as the Banach algebra of all measurable functions $\Phi : \T^n \rightarrow \clb(\cle)$ such that
	\[
	\|\Phi\|_{\infty}=\underset{\z \in \T^n}{\text{ess\,sup}}\|\Phi(\z)\|_{\clb(\cle)} <\infty,
	\]
where $\clb(\cle)$ is the space of all bounded linear operators on $\cle$. For each $\Phi \in L^\infty_{\clb(\cle)}(\T^n)$, the multiplication operator or \emph{Laurent operator} $M_\Phi: L^2_\cle(\T^n) \rightarrow  L^2_\cle(\T^n)$ is defined as
	\[
	M_\Phi f =\Phi f, \quad ( f \in L^2_\cle(\T^n)).
	\]
Moreover, the \textit{Toeplitz operator} $T_\Phi \in \clb(H^2_{\cle}(\D^n)$ is defined as
	\[
	T_\Phi f=P_+M_\Phi f, \quad (f \in H^2_{\cle}(\D^n)),
	\]
where $P_+$ denotes the orthogonal projection of $L^2_\cle(\T^n)$ onto $H^2_\cle(\D^n)$. 
Similarly, the \textit{Hankel operator} $H_\Phi \in \clb(H^2_{\cle}(\D^n))$ is defined by
\begin{equation}\label{eq:classical Hankel operator definition}
		H_\Phi f=P_+ M_\Phi Jf, \quad (f \in H^2_{\cle}(\D^n)),
\end{equation}
	where $J$ is defined on $L^2_\cle(\T^n)$ as
	\[
	(Jf)(\z)=f(\bar{\z}), \quad (f \in L^2_{\cle}(\T^n), \z \in \T^n),
	\]
	where $\bar{\z}=(\bar{z}_1,\dots,\bar{z}_n)$. Note that $J$ is a unitary operator 
	and 
	\[
	J^*=J, \quad J^2=I.
	\]

\begin{rem}
In the polydisc setting, the class of Hankel operators requires some care. In the Hardy space over the unit disc, $H^2_{\cle}(\D)$, a bounded operator $A$ is a Hankel operator if and only if 
\[
T_z^* A = AT_z
\]
However, in the polydisc setting, the corresponding characterization
		\[
		T_{z_i}^*A = AT_{z_i}, \quad \text{for all } i=1,\dots,n,
		\]
is not known to be equivalent to the standard notion of a Hankel operator (in the sense given in \eqref{eq:classical Hankel operator definition}) on $H^2_{\cle}(\D^n)$.
	\end{rem}
	
Motivated by the above intertwining relations, we introduce the following operator class:
\begin{defn}
We call a bounded operator $T$ on $H^2_{\cle}(\D^n)$ to be a \textit{Hankel-type operator} if
		\[
		T_{z_i}^*T=TT_{z_i} \quad \mbox{for all~~~} i = 1,\ldots, n.
		\]
\end{defn}

Note that every Hankel operator $H_{\Phi}$, as defined in \eqref{eq:classical Hankel operator definition}, is a Hankel-type operator.
The usefulness of this definition is justified by the representation obtained below. More precisely, we show that every operator $T$ satisfying the above intertwining relations admits a representation
	\[
	T(p) = P_{+}M_{\Phi}J(p)
	\]
for all analytic polynomials $p$ in $H^2_{\cle}(\D^n)$, where $\Phi$ is an operator-valued strong $H^2$-function, denoted as $H^2_{s}(\T^n,\clb(\cle))$ (see \cite[Appendix A]{raul}). This representation allows us to develop the theory of Toeplitz + Hankel-type operators without requiring any converse characterization of the classical Hankel operators on the polydisc.

\begin{rem}
Let $T$ be a Hankel-type operator on $H^2_{\cle}(\D^n)$. Then $T_{z_i}^*T=TT_{z_i}$ for all $i=1,2,\ldots, n$. For $\bm{k}=(k_1,\ldots, k_n) \in \bm{\Z_{+}}^{n}$, define a map $V_{\bm{k}}: \cle\rightarrow H^2_{\cle}(\D^n)$ by
		\[
		V_{\bm{k}}a=a \z^{\K} \quad (a \in \cle).
		\] 
		Then $V_k$ is an isometry. It is easy to check that
		\[
		V_{\K}^*(a \z^{\AL})=\begin{cases}
			a & \AL=\K\\
			0 & \AL\neq \K
		\end{cases}.
		\]
Then, for $f \in H^2_{\cle}(\D^n)$,  $f=\sum\limits_{\M\in \Z_+^n }V_{\M}^*(f)\z^{\M}$.
Define $\Phi_{\M}= V_{\M}^*TV_{\bm{0}}\in \clb(\cle)$. Let 
\[
\Phi:=\sum\limits_{\M\in \Z_+^n } \Phi_{\M} \z^{\M}. 
\]
Then $\Phi a \in H^2_{\cle}(\D^n)$ for all $a \in \cle$.
Consider
\begin{align*}
			T(a \z^{\K}) = T_{\z}^{*\K} T a
			&= P_{+}\left(\sum\limits_{\M\in \Z_+^n}V_{\M}^*(T a) \z^{\M-\K}\right)\\
			&=P_{+}\left(\sum\limits_{\M\in \Z_+^n}\Phi_m( a)\z^{\M-\K}\right)\\
			&=P_+\left(\Phi (a \z^{-\K}) \right)	\\
			&=P_+M_\Phi J( a \z^{\K}).
\end{align*}
Consequently, for any analytic polynomial $p\in H^2_{\cle}(\D^n)$, we have
\begin{equation}\label{eq:Hankel}
			T(p)= P_+M_\Phi J(p).
\end{equation}		
Thus, for each Hankel-type operator $T$, there exists $\Phi \in H^2_{s}(\T^n,\clb(\cle))$ (see \cite[Appendix A]{raul}) such that (\ref{eq:Hankel}) holds. We will denote $T$ by $H_\Phi$. 
\end{rem}
	
The following results will be useful in the sequel. The Brown-Halmos type characterization for Toeplitz operators on the polydisc setting is as follows (see \cite{Maji}, \cite{ss}).  
		
\begin{thm}
A bounded operator $T$ on $H^2_\cle(\D^n)$ is a Toeplitz operator if and only if
		\[
		T_{z_i}^*TT_{z_i}=T \quad \mbox{for all~~~} i = 1,\ldots, n.
		\]
\end{thm}

We also write $H^{\infty}(\D^n)$ as the Banach algebra of bounded analytic functions on $\D^n$. By means of radial limits, we also identify $H^{\infty}(\D^n)$ as the space $H^2(\D^n)\cap L^{\infty}(\T^n)$. We call a function $\theta\in H^{\infty}(\D^n)$ is \emph{inner} if
\[
|\theta(\z)|=1,\quad \text{a.e. }\z\in\T^n.
\]
	
The notion of Toeplitz and Hankel operators on the Hardy space $H^2(\D)$ is further generalized as  essentially Toeplitz operators, introduced by Barria and Halmos (cf. \cite{Barria}) and essentially Hankel operators, introduced by Avenda\~{n}o (cf. \cite{RMartin}). We are writing the analogous definition for vector-valued Hardy space $H^2_{\cle}(\D)$.

\begin{defn}[Essentially Toeplitz operator]
A bounded operator $T$ on $H^2_{\cle}(\D)$ is said to be \emph{essentially Toeplitz} if $T_z^*TT_z-T$ is compact. The set of all essentially Toeplitz operators is denoted by $\clt_{ess}$.	
\end{defn}

Note that for a finite-dimensional Hilbert space $\cle$, the operator $T_z$ is essentially unitary, i.e., $I-T_zT_z^*$ and $I-T_z^*T_z$ are compact operators. Thus $T$ is essentially Toeplitz if and only if $T_zT-TT_z \in \clk$ if and only if $T_z^*T-TT_z^* \in \clk$, where $\clk$ is the two-sided ideal of all compact operators in $\clb(H^2_{\cle}(\D))$. Furthermore, in that case, $\clt_{ess}$ is a $C^*$-algebra that contains $\clk$.

\begin{defn}[Essentially Hankel operator]
A bounded operator $T$ on $H^2_{\cle}(\D)$ is said to be \emph{essentially Hankel } if $T_z^*T-TT_z$ is compact. We write $\clh_{ess}$ as the set of all essentially Hankel operators on $H^2_{\cle}(\D)$.
\end{defn}
	
For a finite-dimensional Hilbert space $\cle$, an operator $T \in \clb(H^2_{\cle}(\D))$ is essentially Hankel if and only if $T_z^*TT_z^*-T \in \clk$ which is further equivalent to $T_zTT_z-T \in \clk$. Clearly, all compact operators are in $\clh_{ess}$. The following result, concerning the products of essentially Toeplitz and essentially Hankel operators, can be found in \cite{RMartin}:
	
\begin{lem}\label{lem:prod}
Let $T \in \clt_{ess}$ and $H_1, H_2 \in \clh_{ess}$. Then $TH_1, H_1T \in \clh_{ess}$, and $H_1H_2 \in \clt_{ess}$. 
\end{lem}
	
The class that interests us is $\clt_{ess} + \clh_{ess}$. For a finite-dimensional Hilbert space $\cle$, the class $\clt_{ess} + \clh_{ess}$ forms a $C^*$-algebra. The proof can be done on the same lines as that done for $H^2(\D)$ in \cite{RMartin}.

\section{Toeplitz + Hankel-type operators on \texorpdfstring{$H^2_{\cle}(\D^n)$}{H2E(Dn)}}\label{Sec: Toeplitz + Hankel}

The goal of this section is to characterize the Toeplitz + Hankel-type operators on $H^2_{\cle}(\D^n)$.

We need the following lemma to prove the main theorem of this section. The corresponding result in one-variable setting is known; see, for example, \cite{N-DAS-S-DAS-SARKAR, Ehrhardt}. We include a proof for the present polydisc setting for the sake of completeness.
	
\begin{lem}\label{lem: bounded}
Let $\cle$ be a Hilbert space and let $\Phi: \T^n \rightarrow \clb(\cle)$ be an operator-valued measurable function. Suppose that there exist $M>0$ and $j \in \{1,\ldots, n\}$ such that
\[
\| \Phi(\z)(1-z_j^{2m})\|_\infty \leq M
\]
for all $m \geq 1$ and $\z \in \T^n$ a.e. Then, $\Phi \in L^\infty_{\clb(\cle)}(\T^n)$.
\end{lem}

\begin{proof}
For a fixed $j\in\{1, \dots,n\}$, we have
		\[
		\|\Phi(\z)(1-z_j^{2m})\|_\infty \leq M
		\]
for all $m \geq 1$ and a.e.\ $\z \in \T^n$. Consider the set $\Omega_j = \{e^{2\pi i\alpha_j} : \alpha_j\in\mathbb{Q}^c\}$. It is readily seen that one-dimensional normalized Lebesgue measure of $\T\setminus \Omega_j$ is zero. For each $z_j\in \Omega_j$, it follows from the familiar result of Kronecker and Jacobi \cite[Page 38]{Nadkami} that $\{z_j^{2m} : m\in\Z\}$ is dense in $\T$. For $\z\in\T\times\cdots\times \Omega_j \times\cdots\times\T$ a.e., we have
\[
\|\Phi(\z)\|\leq\inf_{m\in\Z}\frac{M}{|1-z_j^{2m}|}=\frac{M}{\sup\limits_{m\in\Z}|1-z_j^{2m}|}=\frac{M}{2}.
\]
Therefore, $\Phi\in L^{\infty}_{\clb(\cle)}(\T^n)$.
\end{proof}
	
We are now ready to prove the main theorem of the section.
	
\begin{thm}
Let $\cle$ be a Hilbert space and let $A\in \clb(H^2_{\cle}(\D^n))$. Then the following are equivalent:
\begin{enumerate}
			\item $A=$ Toeplitz + Hankel-type.
			\item For all $j=1,2,\dots,n$, $T_{z_j}^*AT_{z_j}-A$ is a Hankel-type operator.
			\item For all $j=1,2,\dots,n$, $T_{z_j}^*A-AT_{z_j}$ is a Toeplitz operator.
			\item For all $j,k=1,2,\dots,n$, we have
			\[
			AT_{z_j}+T_{z_j}^*T_{z_k}^*AT_{z_k}=T_{z_j}^*A+T_{z_k}^*AT_{z_j}T_{z_k}.
			\]
\end{enumerate}
\end{thm}

\begin{proof}
We first establish the implication $(1)\implies(4)$. 

Let
			\[
			A=T_{\Phi}+H_{\Psi}
			\]
for some $\Phi\in L^{\infty}_{\clb(\cle)}(\T^n)$ and $\Psi \in H^2_{s}(\T^n,\clb(\cle))$. For all $j,k=1,2,\dots,n$, we have
\begin{align*}
AT_{z_j}+T_{z_j}^*T_{z_k}^*AT_{z_k} 
&= (T_{\Phi}+H_{\Psi})T_{z_j} + T_{z_j}^*T_{z_k}^*(T_{\Phi}+H_{\Psi})T_{z_k}\\
&= T_{\Phi}T_{z_j}+H_{\Psi}T_{z_j} + T_{z_j}^*(T_{z_k}^*T_{\Phi}T_{z_k})+T_{z_j}^*T_{z_k}^*H_{\Psi}T_{z_k}\\
&= T_{\Phi}T_{z_j} + T_{z_j}^*H_{\Psi} + T_{z_j}^*T_{\Phi} + T_{z_k}^*H_{\Psi}T_{z_j}T_{z_k}\\
&= T_{z_j}^*(T_{\Phi}+H_{\Psi}) + T_{\Phi}T_{z_j} + T_{z_k}^*H_{\Psi}T_{z_j}T_{z_k}\\
&=  T_{z_j}^*A +T_{z_k}^*(T_{\Phi}+H_{\Psi})T_{z_j}T_{z_k}\\
&= T_{z_j}^*A +T_{z_k}^*AT_{z_j}T_{z_k}.
\end{align*}

Conversely, assume that (4) holds true. That is, 
			\[
			AT_{z_j}+T_{z_j}^*T_{z_k}^*AT_{z_k}=T_{z_j}^*A+T_{z_k}^*AT_{z_j}T_{z_k}
			 \quad \text{for all } 1\leq j,k\leq n.
			\]
			It follows from the above identity that
			\[
			T_{z_k}^*(T_{z_j}^*A-AT_{z_j})T_{z_k}=T_{z_j}^*A-AT_{z_j} \quad \text{for all } 1\leq j,k\leq n,
			\]
which implies that $T_{z_j}^*A-AT_{z_j}$ is a Toeplitz operator for all $j=1,2,\dots,n$. Therefore, there exists $\Phi_j\in L^{\infty}_{\clb(\cle)}(\T^n)$ such that 
\begin{equation}\label{Eq:Difference is Toeplitz for all j}
T_{z_j}^*A-AT_{z_j}=T_{\Phi_j}, \quad (j=1,2,\dots,n).
\end{equation}
Multiplying \eqref{Eq:Difference is Toeplitz for all j} by $T_{z_j}$ and $T_{z_j}^*$ from right and left, respectively. We obtain
			\[
			T_{z_j}^*AT_{z_j}-AT_{z_j}^2=T_{z_j\Phi_j}
			\]
			and
			\[
			T_{z_j}^{*2}A-T_{z_j}^*AT_{z_j}=T_{\overline{z_j}\Phi_j}.
			\]
			Addition of the above two identites yields
			\[
			T_{z_j}^{*2}A-AT_{z_j}^2=T_{(z_{j}+\bar{z}_j)\Phi_j}=T_{\left(\frac{z_j^2-\bar{z}_j^2}{z_j-\bar{z}_j}\right)\Phi_j}.
			\]
			For each variable $z_j$, we follow the same line of arguments as that of Theorem 2.2 in \cite{N-DAS-S-DAS-SARKAR} and obtain
			\[
			T_{z_j}^{*m}A-AT_{z_j}^m=T_{\left(\frac{z_j^m-\bar{z}_j^m}{z_j-\bar{z}_j}\right)\Phi_j}, \quad (m\geq 1).
			\]
			Note that
			\[
			\left\| \frac{z_j^m-\bar{z}_j^m}{z_j-\bar{z}_j}\Phi_j \right\|_{\infty}=\left\|T_{\left(\frac{z_j^m-\bar{z}_j^m}{z_j-\bar{z}_j}\right)\Phi_j}\right\|=\|T_{z_j}^{*m}A-AT_{z_j}^m\|\leq 2\|A\|.
			\]
			Now Lemma \ref{lem: bounded} infers that
			\[
			\frac{\Phi_j}{\bar{z}_j-z_j}\in L^{\infty}_{\clb(\cle)}(\T^n), \quad (j=1,2,\dots,n).
			\]
			\textsf{Claim:} We prove the following for all $j,k=1,2,\dots,n$:
			\[
			\frac{\Phi_j}{\bar{z}_j-z_j}=\frac{\Phi_k}{\bar{z}_k-z_k}, 
			\]
			i.e.,
			\[
			\Phi_j(\bar{z}_k-z_k)=\Phi_k(\bar{z}_j-z_j), 
			\]
			i.e.,
			\[
			T_{\Phi_j(\bar{z_k}-z_k)}=T_{\Phi_k(\bar{z}_j-z_j)}, 
			\]
			i.e.,
			\[
			T_{z_k}^*T_{\Phi_j}-T_{\Phi_j}T_{z_k}=T_{z_j}^*T_{\Phi_k}-T_{\Phi_k}T_{z_j}.
			\]
\textsf{Proof of the claim:} Using \eqref{Eq:Difference is Toeplitz for all j}, we have the  following:
			\begin{align*}
				T_{z_k}^*T_{\Phi_j}-T_{\Phi_j}T_{z_k}
				&= T_{z_k}^*(T_{z_j}^*A-AT_{z_j})-(T_{z_j}^*A-AT_{z_j})T_{z_k}\\
				&= T_{z_k}^*T_{z_j}^*A-T_{z_k}^*AT_{z_j}-T_{z_j}^*AT_{z_k}+AT_{z_j}T_{z_k}\\
				&= (T_{z_k}^*T_{z_j}^*A - T_{z_j}^*AT_{z_k})-(T_{z_k}^*AT_{z_j}-AT_{z_j}T_{z_k})\\
				&=T_{z_j}^*(T_{z_k}^*A-AT_{z_k})-(T_{z_k}^*A-AT_{z_k})T_{z_j}\\
				&=T_{z_j}^*T_{\Phi_k}-T_{\Phi_k}T_{z_j}
			\end{align*}
			for all $j,k=1,2,\dots,n$. This proves our claim. In that case, let $\Phi\in L^{\infty}_{\clb(\cle)}(\T^n)$ be such that
			\[
			\Phi=\frac{\Phi_j}{\bar{z}_j-z_j}, \quad  (j=1,2,\dots,n).
			\]
			That is,
			\[
			\Phi_j=\Phi(\bar{z}_j-z_j), \quad (j=1,2,\dots,n),
			\]
			which implies
			\[
			T_{\Phi_j}=T_{z_j}^*T_{\Phi}-T_{\Phi}T_{z_j}.
			\]
			Again using \eqref{Eq:Difference is Toeplitz for all j}, we get
\begin{align*}
T_{z_j}^*A-AT_{z_j}&=T_{z_j}^*T_{\Phi}-T_{\Phi}T_{z_j}, \\
\mbox{and hence~~}		
T_{z_j}^*(A-T_{\Phi})&=(A-T_{\Phi})T_{z_j}, \quad (j=1,2,\dots,n).
\end{align*}
			Consequently, $A-T_{\Phi}$ is a Hankel-type operator. That is, there exists $\Psi\in H^2_{s}(\T^n,\clb(\cle))$ such that 
			\[
			A=T_{\Phi}+H_{\Psi}.
			\]
			This establishes the equivalence of (1) and (4). Moreover, the equivalences $(2)\iff(4)$ and $(3)\iff(4)$ follows immediately once we rewrite the condition (4) as follows:
			\[
			T_{z_j}^*(T_{z_k}^*AT_{z_k}-A)=(T_{z_k}^*AT_{z_k}-A)T_{z_j}, \quad (j,k=1,2,\dots,n).
			\]
			which is same as saying that $T_{z_k}^*AT_{z_k}-A$ is a Hankel-type operator for all $k=1,2,\dots,n$. Similarly, we also rewrite condition (4) as follows:
			\[
			T_{z_k}^*(T_{z_j}^*A-AT_{z_j})T_{z_k}=T_{z_j}^*A-AT_{z_j}, \quad (j,k=1,2,\dots,n).
			\]
			which is equivalent to saying $T_{z_j}^*A-AT_{z_j}$ is a Toeplitz operator for all $j=1,2,\dots,n$. 
			
This completes the proof of the theorem.
\end{proof}

\section{Paired operators on \texorpdfstring{$L^2(\T^n)$}{L2Tn}}\label{Sec: Paired operators}

The main aim of this section is to classify paired operators on $L^2(\T^n)$. Recall that $\{\z^{\K} : \K\in\Z^n\}$ is the standard orthonormal basis for $L^2(\T^n)$. We call the elements of $\text{span}\{\z^{\K} : \K\in\Z^n\}$ as the trigonometric polynomials in $L^2(\T^n)$ and denotes this linear space by $\mathbb{P}$. Similarly, we call the members of $\text{span}\{\z^{\K} : \K\in\Z_+^n\}$ and $\text{span}\{\z^{\K} : \K\in\Z^n\setminus \Z_+^n\}$ as the polynomials in $H^2(\D^n)$ and $H^2(\D^n)^{\perp}$, respectively, and denote
	\[
	\mathbb{P}_{+}=\text{span}\{\z^{\K} : \K\in\Z_+^n\},\quad \text{and} \quad 
	\mathbb{P}_{-}=\text{span}\{\z^{\K} : \K\in\Z^n\setminus \Z_+^n\}.
	\]
	It now follows that 
	\begin{equation}\label{Eq: P, P+, and P- closure}
		\overline{\mathbb{P}}=L^2(\T^n),\quad \overline{\mathbb{P}}_{+}=H^2(\D^n),\quad \text{and}\quad \overline{\mathbb{P}}_{-}=H^2(\D^n)^{\perp}.
	\end{equation}
	We also denote $P_{+}=P_{H^2(\D^n)}$, the orthogonal projection of $L^2(\T^n)$ onto $H^2(\D^n)$ and write $P_{-}=I_{L^2(\T^n)}-P_{+}$. Therefore,
	\[
	P_{-}=P_{H^2(\D^n)^{\perp}}.
	\]
	From the fact that $H^2(\D^n)$ is invariant under $M_{z_j}$ for all $j=1,\dots,n$, we observe that
	\begin{equation}\label{Eq: P-MzP+=0}
		P_{+}M_{z_j}P_{+}=M_{z_j}P_{+}, \quad \text{or}\quad P_{-}M_{z_j}P_{+}=0,
	\end{equation}
	for all $j=1,\dots,n$.

The following elementary lemma will be useful in the proof of our main result. Although the result in the case $n=1$ is well known (see \cite[Lemma 8.12]{Fricain-Mashreghi}), we include a proof for the polydisc setting for completeness.

\begin{lem}\label{Lemma: bounded phi lemma}
		Let $\phi\in L^2(\T^n)$ satisfy
		\[
		\|\phi p\|_2\leq C\|p\|_2
		\]
		for all $p\in \mathbb{P}$ and some $C>0$. Then, $\phi\in L^{\infty}(\T^n)$.
\end{lem}

\begin{proof}
		Consider the operator $T: \mathbb{P} \to L^2(\T^n)$ defined by
		\[
		Tp = \phi p, \quad (p\in\mathbb{P}).
		\]
It is evident from the given condition that $T$ is a bounded linear operator from $\mathbb{P}$ into $L^2(\T^n)$. Hence, the operator $T$ admits a unique bounded extension to $\overline{\mathbb{P}}\,( = L^2(\T^n))$, which we continue to denote by $T$. Let $f\in L^2(\T^n)$. Then there exits a sequence $(p_m)_{m\geq 1}$ in $\mathbb{P}$ such that 
		\[
		\|p_m - f\|_2\to 0\quad\text{as $m\to \infty$}.
		\]
		Choose a subsequence $(p_{m_k})$ of $(p_m)$ such that
		\[
		p_{m_k} \to f\quad \text{pointwise a.e. on $\T^n$}.
		\]
		Therefore, 
		\[
		\phi p_{m_k} \to \phi f \quad \text{pointwise a.e. on $\T^n$}.
		\]
		On the other hand, since $T$ is a bounded operator and $p_{m} \to f$ in $L^2(\T^n)$, in particular $p_{m_k} \to f$ in $L^2(\T^n)$, we have
		\[
		\phi p_{m_k} = Tp_{m_k} \to Tf \quad\text{in $L^2(\T^n)$}.
		\]
		Passing, if necessary, to a further subsequence which we continue to denote by $p_{m_k}$, we may assume that
		\[
		\phi p_{m_k} \to Tf\quad \text{and}\quad \phi p_{m_k}\to \phi f\quad \text{pointwise a.e. on $\T^n$}.
		\]
		Therefore,
		\[
		Tf = \phi f \quad \text{a.e. on $\T^n$}.
		\]
		For $d\geq 1$, consider the measurable subset $E_d = \{\z\in\T^n : |\phi(\z)| > d\}$ of $\T^n$. Then
		\[
		\|T\|^2 \mu(E_d) = \|T\|^2\|\chi_{E_d}\|^2_2 \geq \|T\chi_{E_d}\|^2= \|\phi \chi_{E_d}\|^2_2 = \int_{E_d}|\phi|^2 d\mu\geq d^2\mu(E_d)
		\]
for all $d\geq 1$, where $\chi_{E_d}$ denotes the characteristic function. Thus, we must have $\mu(E_d)=0$ for some $d\geq 1$. Consequently, $\phi\in L^{\infty}(\T^n).$
\end{proof}

\begin{thm}\label{Thm: Paired operator theorem on L^2(T^n)}
Let $A\in\clb(L^2(\T^n))$. Then the following are equivalent:
\begin{enumerate}
\item $A$ is a paired operator.
\item $A=M_{z_j}^*AM_{z_j}P_{+}+M_{z_j}AM_{z_j}^*P_{-}$ for all $j=1,\dots,n$.
\end{enumerate}
\end{thm}

\begin{proof}
		Suppose first that $A$ is a paired operator on $L^2(\T^n)$, i.e., 
		\[
		A=S_{\phi, \psi}=M_{\phi}P_{+}+M_{\psi}P_{-}
		\]
		for some $\phi,\psi\in L^{\infty}(\T^n)$. In view of \eqref{Eq: P-MzP+=0}, we note that
		\begin{align*}
			M_{z_j}^*AM_{z_j}P_{+}
			&= M_{z_j}^*S_{\phi,\psi}M_{z_j}P_{+}\\
			&= M_{z_j}^*(M_{\phi}P_{+}+M_{\psi}P_{-})M_{z_j}P_{+}\\
			&= M_{z_j}^*M_{\phi}P_{+}M_{z_j}P_{+} \\
			&= M_{z_j}^*M_{\phi}M_{z_j}P_{+}\\
			&= M_{\phi}P_{+}.
		\end{align*}
		By \eqref{Eq: P-MzP+=0}, we also have $P_{+}M_{z_j}^*P_{-}=0$, and hence
		\begin{align*}
			M_{z_j}AM_{z_j}^*P_{-}
			&= M_{z_j}S_{\phi,\psi}M_{z_j}^*P_{-}\\
			&= M_{z_j}(M_{\phi}P_{+}+M_{\psi}P_{-})M_{z_j}^*P_{-}\\
			&= M_{z_j}M_{\psi}P_{-}M_{z_j}^*P_{-} \\
			&= M_{z_j}M_{\psi}M_{z_j}^*P_{-}\\
			&= M_{\psi}P_{-}.
		\end{align*}
		Evidently, we have
		\[
		A=S_{\phi,\psi}=M_{z_j}^*S_{\phi,\psi}M_{z_j}P_{+} + M_{z_j}S_{\phi,\psi}M_{z_j}^*P_{-}
		,\quad (j=1,\dots,n).
		\]
		Thus, $(1)\implies (2)$ is established.

Conversely, assume that condition (2) holds. That is,
		\begin{equation}\label{Eq: Condition (2) of paired operator theorem}
			A=M_{z_j}^*AM_{z_j}P_{+} + M_{z_j}AM_{z_j}^*P_{-}, \quad (j=1,\dots,n).
		\end{equation}
		Multiplying the above identity by $P_{+}$ from the right yields 
		\[
		AP_{+}=M_{z_j}^*AM_{z_j}P_{+}, \quad (j=1,\dots,n).
		\]
		It follows that
		\begin{equation}\label{Eq: AM-zP+=M-zAP+}
			AM_{z_j}P_{+}=M_{z_j}AP_{+}, \quad (j=1,\dots,n).
		\end{equation}
		Moreover, for all $(m_1,\dots,m_n)\in\Z_{+}^n$ we have 
		\begin{equation}\label{Eq: projection condition for shift invariant}
	P_{+}M_{z_1}^{m_1}\cdots M_{z_n}^{m_n}P_{+}=M_{z_1}^{m_1}\cdots M_{z_n}^{m_n}P_{+}.
		\end{equation}
		
Note that the identities in \eqref{Eq: AM-zP+=M-zAP+} and \eqref{Eq: projection condition for shift invariant} yield, for $\K=(k_1,\dots,k_n)\in\Z_{+}^n\setminus \{(0,0,\dots,0)\}$ (without loss of generality, assume $k_1\geq 1$), that
		\begin{align*}
			A(\z^{\K}) 
			&= AM_{z_1}^{k_1}\cdots M_{z_n}^{k_n}P_{+}(1)\\
			&= AM_{z_1}(M_{z_1}^{k_1-1}\cdots M_{z_n}^{k_n}P_{+})(1)\\
			&= AM_{z_1}(P_{+}M_{z_1}^{k_1-1}\cdots M_{z_n}^{k_n}P_{+})(1)\\
			&= (AM_{z_1}P_{+})M_{z_1}^{k_1-1}\cdots M_{z_n}^{k_n}P_{+}(1)\\
			&= (M_{z_1}AP_{+})M_{z_1}^{k_1-1}\cdots M_{z_n}^{k_n}P_{+}(1)\\
			&= z_1A(P_{+}M_{z_1}^{k_1-1}\cdots M_{z_n}^{k_n}P_{+})(1)\\
			&= z_1AM_{z_1}^{k_1-1}\cdots M_{z_n}^{k_n}P_{+}(1).
		\end{align*}
Hence, the repeated application of identities in \eqref{Eq: AM-zP+=M-zAP+} and \eqref{Eq: projection condition for shift invariant}, we have for $\K=(k_1,\dots,k_n)\in\Z_{+}^n$,
\begin{equation}\label{Eq: A(zk)=zk A(1)}
A(\z^{\K})=z_1^{k_1}\cdots z_n^{k_n}A(1)=\z^{\K}A(1).
\end{equation}
		Set
		\[
		\phi = A(1)\in L^2(\T^n).
		\]
		Therefore, \eqref{Eq: A(zk)=zk A(1)} implies that
		\[
		Aq=\phi q \quad \text{for all } q\in \mathbb{P}_{+}.
		\]
		Assume that $p\in\mathbb{P}$ is arbitrary, then we can choose a $\K\in\Z_{+}^n$ such that $\z^{\K}p\in\mathbb{P}_{+}$. It now follows that
		\[
		\|\phi p\|_2=\|\phi \z^{\K}p\|_2=\|A(\z^{\K}p)\|_2\leq \|A\|\|p\|_2
		\]
for all $p\in\mathbb{P}$. Applying Lemma \ref{Lemma: bounded phi lemma}, we immediately obtain $\phi\in L^{\infty}(\T^n)$. For $f\in H^2(\D^n)$, it follows from the facts recorded in \eqref{Eq: P, P+, and P- closure} that there exists a sequence $(q_m)_{m\geq 1}\subseteq \mathbb{P}_{+}$ such that
		\[
		q_m\to f \quad \text{in } L^2(\T^n).
		\]
Thus,
		\begin{align*}
			\|Af-\phi f\|_2 
			& \leq \|Af-Aq_m\|_2+\|Aq_m-\phi f\|_2\\
			&\leq \|A\|\|f-q_m\|_2+\|\phi q_m-\phi f\|_2\\
			&\leq (\|A\|+\|\phi\|_{\infty})\|q_m-f\|_2\\
			&\longrightarrow 0 \quad \mbox{as~} m \rightarrow \infty.
		\end{align*}
		Hence we obtain
		\[
		Af=\phi f \quad \text{for all }f\in H^2(\D^n),
		\]
		or, 
		\begin{equation}\label{Eq: AP+=M phi P+}
			AP_{+}=M_{\phi}P_{+}. 
		\end{equation}
		Again, we observe from \eqref{Eq: Condition (2) of paired operator theorem} that $AP_{-}=M_{z_j}AM_{z_j}^*P_{-}$, i.e.,
		\begin{equation}\label{Eq: A M-z* P- = M-z* A P-}
			AM_{z_j}^*P_{-}=M_{z_j}^*AP_{-},\quad (j=1,\dots,n).
		\end{equation}
		Furthermore, we note that for all $(m_1,\dots,m_n)\in\Z^n_{+}$,
		\begin{equation}\label{Eq: projection condition for shift co-invariant}
			P_{-}M_{z_1}^{*m_1}\cdots M_{z_n}^{*m_n}P_{-}=M_{z_1}^{*m_1}\cdots M_{z_n}^{*m_n}P_{-}.
		\end{equation} 
		For each $j=1,\dots,n$ and $k_j\geq 2$, we have
		\begin{align*}
			A(\bar{z}_j^{k_j})
			&= AM_{z_j}^{*k_j-1}P_{-}(\bar{z}_j)\\
			&= AM_{z_j}^*M_{z_j}^{*k_j-2}P_{-}(\bar{z}_j)\\
			&= AM_{z_j}^*(P_{-}M_{z_j}^{*k_j-2}P_{-})(\bar{z}_j) 
			\tag{using \eqref{Eq: projection condition for shift co-invariant}}\\
			&= (AM_{z_j}^*P_{-})(M_{z_j}^{*k_j-2}P_{-})(\bar{z}_j)\\
			&= (M_{z_j}^*AP_{-})(M_{z_j}^{*k_j-2}P_{-})(\bar{z}_j)
			\tag{using \eqref{Eq: A M-z* P- = M-z* A P-}}\\
			&= \bar{z}_jA(P_{-}M_{z_j}^{*k_j-2}P_{-})(\bar{z}_j)\\
			&= \bar{z}_j AM_{z_j}^{*k_j-2}P_{-}(\bar{z}_j).
		\end{align*}
		Therefore repeatedly applying \eqref{Eq: A M-z* P- = M-z* A P-} and \eqref{Eq: projection condition for shift co-invariant}, we obtain
\begin{equation}\label{Eq: A(zj bar ^ k) = zj bar ^k (zjA(zj))}
A(\bar{z}_j^{k_j})=\bar{z}_j^{k_j-1}A(\bar{z}_j)=\bar{z}_j^{k_j}(z_jA(\bar{z}_j))
\end{equation}
which holds for all $k_j\geq 1$.
		Set 
		\[
		\psi_j=z_jA(\bar{z}_j)\in L^2(\T^n),\quad (j=1,\dots,n).
		\]

		\textsf{We now claim the following:}
		\begin{enumerate}[label=(\alph*)]
			\item $\psi_1=\psi_2=\dots=\psi_n(=\psi, \text{ say})$.
			\item $Ap=\psi p$ for all $p\in\mathbb{P}_{-}$.
			\item $\psi\in L^{\infty}(\T^n)$.
		\end{enumerate}
		It follows from \eqref{Eq: Condition (2) of paired operator theorem} that
		\[
		A(\bar{z}_j)=0 + M_{z_k}AM_{z_k}^*P_{-}(\bar{z}_j)=M_{z_k}AM_{z_k}^*M_{z_j}^*(1),\quad (j,k=1,\dots,n),
		\]
		so that 
		\[
		z_jA(\bar{z}_j)=M_{z_j}M_{z_k}AM_{z_k}^*M_{z_j}^*(1),\quad (j,k=1,\dots,n).
		\]
		Therefore,
		\[
		z_jA(\bar{z}_j)=z_kA(\bar{z}_k),\quad \text{or}\quad \psi_j=\psi_k,\quad (j,k=1,\dots,n).
		\]
		This proves claim (a). Thus, from \eqref{Eq: A(zj bar ^ k) = zj bar ^k (zjA(zj))} we have 
		\[
		A(\bar{z}_j^{k_j})=\psi \bar{z}_j^{k_j},\quad (k_j\geq 1, j=1,\dots,n).
		\]
		To establish claim (b), suppose now that $\emptyset \neq E\subseteq \{1,\dots,n\}$ and consider the following set
		\[
		\cls_E=\{\K=(k_1,\dots,k_n)\in\Z^n : k_j<0 \text{ if } j\in E \text{ and } k_{j}\geq 0 \text{ if } j\notin E\}.
		\]
		Clearly,
		\[
		\Z^n\setminus\Z^n_{+}=\bigcup_{\emptyset \neq E\subseteq \{1,\dots,n\}}\cls_E.
		\]
First, using \eqref{Eq: Condition (2) of paired operator theorem}, we observe the following interesting fact that
		\[
		A(z_j\z^{\K})=z_jA(\z^{\K}),
		\]
for all $1\leq j\leq n$ and $\K\in\Z^n$ as long as $z_j\z^{\K}\in \mathbb{P}_{-}$.

Now fix an arbitrary $\emptyset\neq E\subseteq\{1,\dots,n\}$ and $\K=(k_1,\dots,k_n)\in \cls_E$ so that $\z^{\K}\in \mathbb{P}_{-}$, then from the repeated application of \eqref{Eq: Condition (2) of paired operator theorem} and the fact above, we obtain 
		\begin{align}
			A(\z^{\K}) 
			&= A\left(\prod_{j\notin E}z_j^{k_j}\,\prod_{j\in E}z_j^{k_j}\right)\notag \\
			&= \left(\prod_{j\notin E}z_j^{k_j}\right)A\left(\prod_{j\in E}z_j^{k_j}\right)\label{Eq: A(z^k) = (prod not E)(prod E)}.
		\end{align}
		Moreover, if $\bm{m}=(m_1,\dots,m_n)$ be such that $m_j\leq 0$ for all $j=1,\dots,n$. Then for each $j$, we have 
		\[
		A(\bar{z}_j\z^{\bm{m}})=AM_{z_1}^{*l_1}\cdots M_{z_n}^{*l_n}P_{-}(\bar{z}_j), \quad \text{where}\quad (l_j=-m_j,\,\forall j=1,\dots,n).
		\]
Once again, the repeated application of the identities in \eqref{Eq: A M-z* P- = M-z* A P-} and \eqref{Eq: projection condition for shift co-invariant} reduces the above identity as follows
		\begin{align*}
			A(\bar{z}_j \z^{\bm{m}})
			&= M_{z_1}^{*l_1}\cdots M_{z_n}^{*l_n}AP_{-}(\bar{z}_j)\\
			&= \z^{\bm{m}}A(\bar{z}_j)\\
			&= \bar{z}_j\z^{\bm{m}}(z_jA(\bar{z}_j))\\
			&= (\bar{z}_j \z^{\bm{m}})\psi.
		\end{align*}
		Consequently, we deduce that
		\[
		A\left(\prod_{j\in E}z_j^{k_j}\right)=\left(\prod_{j\in E}z_j^{k_j}\right)\psi.
		\]
		Hence, the equation \eqref{Eq: A(z^k) = (prod not E)(prod E)} becomes
		\[
		A(\z^{\K})=\left(\prod_{j\notin E}z_j^{k_j}\right)\left(\prod_{j\in E}z_j^{k_j}\right)\psi=\z^{\K}\psi
		\]
		for all $\K\in\cls_E$ and $\emptyset\neq E\subseteq\{1,\dots,n\}$. Thus,
		\[
		A(\z^{\K})=\z^{\K}\psi, \quad (\K\in\Z^n\setminus\Z^n_{+}).
		\]
		It now immediately follows that 
		\[
		Ap=\psi p, \quad (p\in \mathbb{P}_{-}).
		\]
		This proves our claim (b). Again, we assume that $p\in\mathbb{P}$ is arbitrary, then we can choose a $\K=(k_1,\dots,k_n)\in\Z^n$ such that $\z^{\K}p\in\mathbb{P}_{-}$ so that
		\[
		\|\psi p\|_2=\|\psi \z^{\K}p\|_2=\|A(\z^{\K}p)\|_2\leq \|A\|\|p\|_2
		\]
		for all $p\in\mathbb{P}$. Again, the application of Lemma \ref{Lemma: bounded phi lemma} immediately yields $\psi\in L^{\infty}(\T^n)$ and the claim (c) is proved. Finally, for $f\in H^2(\D^n)^{\perp}$, \eqref{Eq: P, P+, and P- closure} implies that there exists a sequence $(p_m)_{m\geq 1}\subseteq \mathbb{P}_{-}$ such that 
		\[
		p_m\to f \quad \text{in } L^2(\T^n).
		\]
		Thus,
		\begin{align*}
			\|Af-\psi f\|_2 
			& \leq \|Af-Ap_m\|_2+\|Ap_m-\psi f\|_2\\
			&\leq \|A\|\|f-p_m\|_2+\|\psi p_m-\psi f\|_2\\
			&\leq (\|A\|+\|\psi\|_{\infty})\|p_m-f\|_2
			\longrightarrow 0 \quad \mbox{as~} m \rightarrow \infty.
		\end{align*}
		Hence, we obtain
		\[
		Af=\psi f \quad \text{for all }f\in H^2(\D^n)^{\perp},
		\]
		or, 
		\begin{equation}\label{Eq: AP-=M psi P-}
			AP_{-}=M_{\psi}P_{-}. 
		\end{equation}
Finally, combining \eqref{Eq: AP+=M phi P+} and \eqref{Eq: AP-=M psi P-}, we obtain 
		\[
		A=M_{\phi}P_{+}+M_{\psi}P_{-}=S_{\phi, \psi}
		\]
		for some $\phi,\psi\in L^{\infty}(\T^n)$ and this establishes the implication $(2)\implies (1)$. 
		
This completes the proof of the theorem.
\end{proof}

We now turn our attention to generalizing the concept of $\theta$-paired operators on $H^2(\D^n)$. Recently, the characterization of $\theta$-paired operators on $H^2(\D)$ was given in \cite{N-DAS-S-DAS-SARKAR}.  Recall that  for a non-constant inner function $\theta \in H^\infty(\D^n)$, $\theta H^2(\D^n)$ is a nonzero closed subspace of $H^2(\D^n)$ which is invariant under $T_{z_i}$ for all $i\in \{1,\ldots, n\}$. Therefore,	
	\[
	H^2(\D^n)=\theta H^2(\D^n) \oplus K_\theta,
	\]
where $K_\theta=H^2(\D^n)\ominus \theta H^2(\D^n)$.

\begin{defn}
Let $\theta \in H^\infty(\D^n)$ be a non-constant inner function. The $\theta$-paired operator on $H^2(\D^n)$ with symbols $\phi, \psi \in H^\infty(\D^n)$ is defined as
		\[
		S_{\phi,\psi}^\theta=T_\phi P_{\theta H^2(\D^n)} +T_\psi P_{K_\theta}.
		\]
\end{defn}
	
We now state the characterization result for $\theta$-paired operators on $H^2(\D^n)$. The proof follows along the same lines as in the case of $H^2(\D)$ in \cite{N-DAS-S-DAS-SARKAR}, so we omit it.
	
\begin{thm}
Let $\theta\in H^{\infty}(\D^n)$ be a non-constant inner function. A bounded operator $X$ on $H^2(\D^n)$ is a $\theta$-paired operator if and only if the following conditions are satisfied:
\begin{enumerate}
			\item $\theta H^2(\D^n)$ is invariant under $X$.
			\item There exist a non-zero $f_0 \in K_\theta$ and $v \in H^\infty(\D^n)$ such that
			\[
			Xf_0=vf_0
			\]
			and
			\[
			XT_{z_i} - T_{z_i}X =(X-T_v)P_{\theta H^2(\D^n)}T_{z_i}P_{K_\theta}
			\]
			for all $i=1,\ldots, n.$
\end{enumerate}
\end{thm}

\section{Essentially Toeplitz + Essentially Hankel operators}\label{Sec:ET+EH}
	
In this section, we will be giving an intrinsic characterization for an operator to be written as a sum of an essentially Toeplitz and an essentially Hankel operator. Before stating the main result, we give some lemmas, which we will use in our proof. The following result is given in \cite[Chapter 2]{Nadkami}:
	
\begin{lem}(Von-Neumann ergodic theorem, \cite{Nadkami})\label{lem:ergodic}
Let $U: \cll \rightarrow \cll$ be a unitary operator on a separable Hilbert space $\cll$. Then for every $x \in \cll$,
		\[
		\lim\limits_{n\rightarrow \infty}\frac{1}{n}\sum\limits_{k=0}^{n-1} U^{k}x=P_{\clm_U}x,
		\]
where $P_{\clm_U}$ is the orthogonal projection of $\cll$ onto the closed subspace $\clm_U=\{v \in \cll: Uv=v\}$.
\end{lem}
	
The following result is given by Olsen and Pederson for any $\sigma$-unital C$^*$-algebra $\cla$ and its \emph{corona C$^*$-algebra} $\clc(\cla)$ (see \cite{Olsen} for details):
	
\begin{prop}[\cite{Olsen}]\label{prop:corona}
Let $\cla$ be a $\sigma$-unital C$^*$-algebra and $a,b$ be elements in $\clc(\cla)$ with $a^*a \leq b^*b$. If $\cld$ is a separable subset of $\clc(\cla)$ whose every element commutes with $a,b^*$, and $b^*b$, then $a=cb$ for some $c \in \clc(\cla)$ such that $c$ commutes with every element of $\cld$.
\end{prop}
	
The following result related to the theory of C$^*$-algebras is well-known and can be found in the standard references (see \cite{ConwayOT, Davidson}):
	
\begin{prop}\label{prop:injective}
Let $\cla$ and $\clb$ be two C$^*$-algebras and $\phi:\cla \rightarrow \clb$ be a $*$-homomorphism with trivial kernel. Then, for self-adjoint $a \in \cla$, $a\geq 0$ if and only if $\phi(a)\geq 0$.
\end{prop}

\begin{proof}
The necessary part is easy to verify. For the converse, let $\phi(a) \geq 0$. Note that $\phi(\cla)$ is also a C$^*$-algebra, there exists $b \in \cla$ such that $\phi(a)=\phi(b)^*\phi(b)=\phi(b^*b)$. Since $\phi$ has trivial kernel, we get $a=b^*b$. This finishes the proof.
\end{proof}

For the rest of the section, we will assume $\cle$ to be a finite-dimensional complex Hilbert space.  We are now ready to state the main result of this section.
	
\begin{thm}\label{thm:maintheorem}
		Let $A \in \clb(H^2_{\cle}(\D))$. Then the following are equivalent:
		\begin{enumerate}
			\item $A \in \clt_{ess}+\clh_{ess}$.
			\item $T_z^*AT_z-A$ is an essentially Hankel operator.
			\item $T_z^*A-AT_z$ is an essentially Toeplitz operator.
			\item $AT_z+T_z^{*2}AT_z-T_z^*A-T_z^*AT_z^2$ is compact.
		\end{enumerate}
\end{thm}
	
We will prove the theorem in two parts. First we establish the equivalence of statements $(2), (3)$, and $(4)$ and later we prove $(1)\iff (3)$. The former part follows from the definitions only but we add the proof for the sake of completeness. Also, for the rest of the section we will denote the unilateral shift $T_z$ by $S$ for the simplicity of notation.
	
\begin{prop}
		Let $A \in \clb(H^2_\cle(\D))$. Then the following are equivalent:
		\begin{enumerate}[label=(\roman*)]
			\item $S^*AS-A$ is an essentially Hankel operator.
			\item $S^*A-AS$ is an essentially Toeplitz operator.
			\item $AS+S^{*2}AS-S^*A-S^*AS^2$ is compact.
		\end{enumerate}
\end{prop}

\begin{proof}
Note that $S^*AS-S \in \clh_{ess}$ if and only if there exists $K \in \clk$ such that
\begin{align*}
			&S^*(S^*AS-A)-(S^*AS-A)S=K\\
			\iff & AS+S^{*2}AS-S^*A-S^*AS^2=K\\
			\iff & S^*(S^*A-AS)S - (S^*A-AS) = K\\
			\iff & S^*A-AS \in \clt_{ess}.
\end{align*}
This establishes the equivalence of $(i), (ii)$, and $(iii)$.
\end{proof}

For the implication $(1)\iff (3)$ of Theorem \ref{thm:maintheorem}, we need some prerequisites (cf. \cite{Mustafayev}) :
\begin{rem}
Let $\mathit{c}$ be the space of convergent sequences of complex numbers. Define a linear functional $\eta$ on $\mathit{c}$ by
\[
\eta(\{x_n\})=\lim_{n\rightarrow \infty} x_n \qquad (\{x_n\} \in \mathit{c}).
\]
\end{rem}
Using the Hahn-Banach theorem, $\eta$ can be extended to $l^\infty$ such that it holds certain properties. The following result can be found in \cite[Chapter III]{Conway}:

\begin{lem}[\cite{Conway}]
There exists a linear functional $\eta: l^\infty \rightarrow \C$ such that
\begin{enumerate}
			\item $\|\eta\|=1$.
			\item If $\{x_n\}\in \mathit{c}$, $\eta(\{x_n\})
			=\lim\limits_{n\rightarrow \infty} x_n.$
			\item If $\{x_n\} \in l^\infty$ such that $x_n \geq 0$ for all $n$, 
			then $\eta(\{x_n\})\geq 0$.
			\item  If $\{x_n\} \in l^\infty$ and  $y_n=x_{n+1}$, then 
			$\eta(\{x_n\})=\eta(\{y_n\})$.
\end{enumerate}
\end{lem}
	
The linear functional $\eta$ on $l^\infty$ is called a \emph{Banach limit}. We will be using this notion of Banach limit in the following setup. One can refer to \cite{Mustafayev} for details.
	
	For a Hilbert space $\cll$, let $\cll_0$ be the space of all weakly null sequences 
	$\{x_n\}$ in $\cll$. Define a semi-inner product $p$ in $\cll_0$ by
	\[
	p(\{x_n\},\{y_n\}) = \text{l.i.m.}_n\left\langle x_n,y_n\right\rangle, \quad (\{x_n\},\{y_n\}\in\cll_0),
	\]
	where l.i.m. is a fixed Banach limit. For
	\[
	\cln=\left\{\{x_n\} \in \cll_0 : p(\{x_n\},\{x_n\})=0\right\},
	\]
	we equip the quotient $\cll_0/\cln$ with the following inner product, defined by:
	\[
	\left\langle \{x_n\}+ \cln,\{y_n\}+ \cln \right\rangle=\text{l.i.m.}_n\left\langle x_n,y_n\right\rangle.
	\]
\noindent
Let $\widehat{\cll}$ be the completion of $\cll_0/\cln$ with respect to the above inner product. For $T \in \clb(\cll)$, define an operator $\widehat{T}$ on $\widehat{\cll}$ by
	\[
	\widehat{T}:\{x_n\}+ \cln \mapsto\{T x_n\}+ \cln.
	\]
	Moreover,
	\[
	\left\|\widehat{T}(\{x_n\}+ \cln )\right\|^2=\text{l.i.m.}_n \|Tx_n\|^2.
	\]
	The operator $\widehat{T}$ is called the \emph{limit operator} associated with $T$.

\begin{lem}[\cite{Mustafayev}]\label{Lem:Properties}
If $\widehat{T}$ is the limit operator associated with $T$, then
\begin{enumerate}
			\item $\|\widehat{T}\|\leq \|T\|.$
			\item The map $T \mapsto \widehat{T}$ is a $*$-homomorphism.
			\item T is compact if and only if $\widehat{T}=0$.
			\item $T$ is essentially isometric, i.e., $I-T^*T$ is compact if and only if $\widehat{T}$ is an isometry.
\end{enumerate}
\end{lem}
	
For the rest of the section, we will denote, for $X,Y \in \clb(\cll)$, 
	\[
	\cli(X)=\{A \in \clb(\cll): AX=XA\}
	\quad \text{and} \quad
	\cli(X,Y)=\{A \in \clb(\cll): XA=AY\}.
	\]
	
The following lemma is easy to verify. 
	
\begin{lem}\label{lem:ET}
For $T, H \in \clb(H^2_\cle(\D))$, $T$ and $H$ are essentially Toeplitz and essentially Hankel operators, respectively, if and only if  $\wh{T} \in \cli(\widehat{S})$ and $\wh{H} \in \cli(\wh{S}^*,\wh{S})$.
\end{lem}
	
It is worthy to mention that, since $S$  is essentially unitary on $H^2_\cle(\D)$, $\wh{S}$ is unitary on $\wh{H^2_\cle(\D)}$. Now we are ready to prove the equivalence $(1)\iff (3)$ of Theorem \ref{thm:maintheorem}:

\begin{prop}
Let $A \in \clb(H^2_\cle(\D))$. Then $A \in \clt_{ess} + \clh_{ess}$ if and only if 
$S^*A-AS \in \clt_{ess}$.  
\end{prop}

\begin{proof}
Let $A \in \clt_{ess} + \clh_{ess}$. Then there exists $H \in \clh_{ess}$ such that  $A-H =T\in \clt_{ess}$. Thus
\begin{align*}
			&S^*(A-H)-(A-H)S =S^*T-TS\\
			\implies & S^*A-AS=S^*H-HS+S^*T-TS.  
\end{align*}
Since $S^*H-HS \in \clk \subset \clt_{ess}$ and $S^*T-TS \in \clt_{ess}$, we get
$S^*A-AS \in\clt_{ess}$.

Conversely, if $S^*A-AS=T \in \clt_{ess}$, then by Lemma \ref{Lem:Properties} (2) and Lemma \ref{lem:ET},
\begin{equation}\label{eq:n=1}
\wh{S}^*\wh{A}-\wh{A}\wh{S}=\wh{T},
\end{equation}
where $\wh{T}\in \cli(S)$.  Note that if $A$ is compact, then the result holds trivially.  Additionally, if $T$ is compact, then $A \in \clh_{ess}$. Again we are done. Thus, we can assume that $A$ and $T$ are non-compact so that $\wh{A}$ and $\wh{T}$ are non-zero operators. Now, post-multiplying (\ref{eq:n=1}) by $\wh{S}$ and pre-multiplying by 
$\wh{S}^*$, respectively gives
\begin{align*}
\wh{S}^*\wh{A}\wh{S}-\wh{A}\wh{S}^2&=\wh{T}\wh{S}=\wh{S}\wh{T}\\
\wh{S}^{*2}\wh{A}-\wh{S}^*\wh{A}\wh{S}&=\wh{S}^*\wh{T}.
\end{align*}
Adding both the equations, we get
\begin{equation}\label{eq:n=2}
\wh{S}^{*2}\wh{A}-\wh{A}\wh{S}^2=(\wh{S}+\wh{S}^*)\wh{T}.
\end{equation}
Again, post-multiplying (\ref{eq:n=2}) by $\wh{S}$, pre-multiplying (\ref{eq:n=1}) by 
$\wh{S}^{*2}$ and adding them together yield
\[
\wh{S}^{*3}\wh{A}-\wh{A}\wh{S}^3=\left(I+ \wh{S}^2+\wh{S}^{*2}\right)\wh{T}.
\]
		
We claim that 
		\[
		\wh{S}^{*n}\wh{A}-\wh{A}\wh{S}^n=
		\begin{cases}
			\sum\limits_{k=1}^{n/2}\left(\wh{S}^{2k-1}+\wh{S}^{*2k-1}\right)\wh{T}, & \text{ for }n \text{ even}\\  
			\wh{T}+\sum\limits_{k=1}^{(n-1)/2}\left(\wh{S}^{2k}+\wh{S}^{*2k}\right)\wh{T}, & \text{ for } n \text{ odd}.
		\end{cases}
		\]
		
We will prove the claim by induction and we will give the proof for $n$ even,  the odd case follows analogously.
		
Let the claim be true for $n-1$, which is odd. Consider
\begin{align*}
			\wh{S}^{*n}\wh{A}-\wh{A}\wh{S}^n &=\wh{S}^{*n}\wh{A}-\wh{S}^{*{n-1}}\wh{A}\wh{S}+\wh{S}^{*{n-1}}\wh{A}\wh{S}-\wh{A}\wh{S}^n\\
			&=\wh{S}^{*n-1}\left(\wh{S}^*\wh{A}-\wh{A}\wh{S}\right)+ \left(\wh{S}^{*n-1}\wh{A}-\wh{A}\wh{S}^{n-1}\right)\wh{S} \\
			&=\wh{S}^{*n-1}\wh{T}+\wh{T}\wh{S}+\sum\limits_{k=1}^{(n-2)/2}\left(\wh{S}^{2k}+\wh{S}^{*2k}\right)\wh{T}\wh{S}\\
			&=\sum\limits_{k=1}^{n/2}\left(\wh{S}^{2k-1}+\wh{S}^{*2k-1}\right)\wh{T},
\end{align*}
where the last equality holds because of the fact that $\wh{T}\in \cli(\wh{S})$. This proves the claim. As a consequence, we get, for each $n$,
\begin{equation}\label{eq:uniform1}
\left\|\sum\limits_{k=1}^n \left(\wh{S}^{2k-1}+\wh{S}^{*2k-1}\right)\wh{T} \right\| \leq \|\wh{S}^{*2n}\wh{A}-\wh{A}\wh{S}^{2n}\| \leq 2\|\wh{A}\|,
\end{equation}
and
\begin{equation}\label{eq:uniform2}
\left\|\wh{T}+\sum\limits_{k=1}^{n}\left(\wh{S}^{2k}+\wh{S}^{*2k}\right)\wh{T}\right\| \leq 2\|\wh{A}\|.
\end{equation}

Define $R=S^*-S$. Our aim is to prove that $T=T'R+K_1$, where $T'$ is an essentially Toeplitz operator and $K_1$ is compact. We will complete the proof through a series of claims:\vspace{5pt}

\textsf{Claim 1:} $\cln(\wh{R}) \subseteq \cln(\wh{T})$.\\
Let $x$ be a unit vector in $\cln(\wh{R})$, that is, $\wh{S}x=\wh{S}^*x$. Pre-multiplying by $\wh{S}$ yields $\left(\wh{S}^2-I\right)x=0$.  Then (\ref{eq:uniform2}) implies
\begin{align*}
\left\|\wh{T}x + \sum\limits_{k=1}^n 2\wh{T}\wh{S}^{2k}x \right\| &\leq 2\|\wh{A}\|\\
\left\|\wh{T}x +2\sum\limits_{k=1}^n \wh{T}x \right\| & \leq 2\left\|\wh{A}\right\|\\
(2n+1)\left\|\wh{T}x\right\| &\leq 2\left\|\wh{A}\right\|
\end{align*}
This is true for all $n$ and also $\wh{A}$ is assumed to be non-zero, thus $\wh{T}x =0$.\vspace{5pt}

\textsf{Claim 2:} $\wh{T}^*\wh{T} \leq \lambda^2 \wh{R}^*\wh{R}$ for some $\lambda \in \mathbb{R}$.

Let $\wh{R}_n=\sum\limits_{k=1}^n\left(\wh{S}^{*2k-1}+\wh{S}^{2k-1}\right)$. Then 
		\[
		\wh{R_n}\wh{R}=\wh{S}^{*2n}-\wh{S}^{2n}
		\]
		which implies 
\begin{equation}\label{eq:mean}
\frac{1}{n}\sum\limits_{k=1}^n \wh{R_k}\wh{R}\wh{T}=\frac{1}{n}\sum\limits_{k=1}^n \left(\wh{S}^{*2k}-\wh{S}^{2k}\right)\wh{T}.
\end{equation}
Firstly, for $x \in \wh{H^2_\cle(\D)}$, by (\ref{eq:uniform1}),
\begin{equation}\label{eq:LHS}
\frac{1}{n}\sum\limits_{k=1}^n \left\|\wh{R_k}\wh{R}\wh{T}x\right\|^2\leq \frac{1}{n} \sum\limits_{k=1}^n 4\left\|\wh{A}\right\|^2\left\|\wh{R}x\right\|^2= 4\left\|\wh{A}\right\|^2\left\|\wh{R}x\right\|^2.
\end{equation}
Now, consider
\begin{align}
\frac{1}{n}\sum\limits_{k=1}^n \left\| \left(\wh{S}^{*2k}-\wh{S}^{2k}\right)\wh{T}x\right\|^2
&= -\frac{1}{n} \sum\limits_{k=1}^n \left\langle  (\wh{S}^{*2k}-\wh{S}^{2k})^2\wh{T}x, \wh{T}x\right\rangle \notag\\
&= -\frac{1}{n}\sum\limits_{k=1}^n\left\langle \left(\wh{S}^{*4k}+\wh{S}^{4k}-2I\right)\wh{T}x,\wh{T}x\right\rangle \notag \\
&=2\left\|\wh{T}x\right\|^2-\frac{1}{n} \sum\limits_{k=1}^n\left\langle \left(\wh{S}^{*4k}+\wh{S}^{4k}\right)\wh{T}x,\wh{T}x\right\rangle.\label{eq:PM commutes with T hat - 1}
\end{align}
Note that $\clm_{\wh{S}^4}=\{u \in \wh{H^2_\cle(\D)}: \wh{S}^4 u=u\}=\clm_{\wh{S}^{*4}}$. Then $P_{\clm_{\wh{S}^4}}=P_{\clm_{\wh{S}^{*4}}}$. We will denote it by $\clm$. Thus by Lemma \ref{lem:ergodic},
\[
\lim\limits_{n\rightarrow \infty}\left\|\frac{1}{n}\sum\limits_{k=1}^n \wh{S}^{4k}\wh{T}x-P_{\clm}\wh{T}x  \right\|=  \lim\limits_{n\rightarrow \infty}\left\|\frac{1}{n}\sum\limits_{k=1}^n \wh{S}^{*4k}\wh{T}x-P_{\clm}\wh{T}x  \right\|= 0.
\]
Thus
\begin{equation}\label{eq:PM commutes with T hat - 2}
\frac{1}{n}\left\| \sum\limits_{k=1}^n \left(\wh{S}^{*4k}+\wh{S}^{4k}\right)\wh{T}x  \right\| \rightarrow 2\left\|P_\clm \wh{T}x\right\| \quad \text{ as } {n \rightarrow \infty}.
\end{equation}
Note that $\clm$ reduces $\widehat{T}$ so that $P_\clm$ commutes with $\wh{T}$, therefore if $x \in \clm^\perp$, \eqref{eq:PM commutes with T hat - 1} and \eqref{eq:PM commutes with T hat - 2} together yield
		\[
		\frac{1}{n}\sum\limits_{k=1}^n \left\| \left(\wh{S}^{*2k}-\wh{S}^{2k}\right)\wh{T}x\right\|^2\rightarrow 2\left\|\wh{T}x\right\|^2.
		\]
In turn, (\ref{eq:mean}) and (\ref{eq:LHS}) imply
\begin{equation}\label{eq:1}
\left\|\wh{T}x\right\|^2 \leq 2\left\|\wh{A}\right\|^2\left\|\wh{R}x\right\|^2, \quad  (x \in \clm^\perp).
\end{equation}

We are left with the case when $x \in \clm$. By Claim 1, it is enough to prove that 
\[
\left\| \wh{T}x\right\|^2 \leq \lambda^2 \left\| \wh{R}x \right\|^2 \quad (x \in \clr(\wh{R})).
\]
For $x \in \clr(\wh{R})\cap \clm$,  
		\[
		0=\left(\wh{S}^4-I\right)x=\left(\wh{S}^2-I\right)\left(\wh{S}^{2}+I\right)x
		\]
which guarantees $\left(\wh{S}^{2}+I\right)x=0$. Indeed, as $\left(\wh{S}^{2}+I\right)x \in \cln\left(\wh{R}\right)\cap \clr\left(\wh{R}\right).$ 		
Now
\[
\wh{R}x=\left(\wh{S}^*-\wh{S}\right)x=2\wh{S}^*x,
\]
and hence 
\begin{equation}\label{eq:2}
\left\|\wh{T}x\right\|^2\leq \left\|\wh{T}\right\|^2\|x\|^2=\frac{1}{4}\left\|\wh{T}\right\|^2\left\|\wh{R}x\right\|^2.
\end{equation}
Combining (\ref{eq:1}) and (\ref{eq:2}), we get
		\[
		\left\|\wh{T}x\right\|^2\leq \lambda^2\left\|\wh{R}x\right\|^2 \quad \text{ for all }x \in \wh{H^2_\cle(\D)},
		\]
		where $\lambda= \max\left\{\frac{\left\|\wh{T}\right\|}{2},\sqrt{2}\left\|\wh{A}\right\|\right\}$. This establishes the claim.\vspace{5pt}

\textsf{Claim 3:} $T=(S^*-S)T'+K_1$ where $T'$ is an essential Toeplitz operator and $K_1$ compact.
		
Denote $\clc(\clk)=\clb(H^2_\cle(\D))/\clk$ as the Calkin algebra of $H^2_\cle(\D)$,
and define a map $\Phi:\clc(\clk) \rightarrow \wh{H^2_\cle(\D)}$ by
		\[
		\Phi(T +\clk):=\wh{T}, \quad (T \in \clb(H^2_\cle(\D))).
		\]
Then it is easy to verify that $\Phi$ is a $*$-homomorphism. Moreover, $\Phi$ has trivial kernel. As if, $T+\clk \in \ker(\Phi)$, then $\wh{T}=0$. By Lemma \ref{Lem:Properties}, $T$ is compact. Now Claim 2 yields $\Phi(T^*T+\clk) \leq \lambda^2 \Phi(R^*R +\clk)$. Since $\Phi$ has trivial kernel, by Proposition \ref{prop:injective}, we get		
		\[
		T^*T+\clk \leq \lambda^2 R^*R +\clk.
		\]
Using Proposition \ref{prop:corona} on the corona C$^*$-algebra $\clc(\clk)$ of $\clk$ and let $\cld=\{S+\clk\}$, we get that there exist $T'+\clk$ commuting with $S+\clk$ such that
\[
T+\clk=(T'+\clk)(R+\clk),
\]
that is, there exist $K_1 \in \clk$ such that
\[
T= T'R +K_1.
\]
It is a routine check that $T'+\clk$ commuting with $S+\clk$ implies $T'S-ST' \in \clk$, i.e., $T'$ is an essentially Toeplitz operator. This proves the claim.

Now, plugging the value of $T$, and using the fact that $T'\in \clt_{ess}$, we get
		\begin{align*}
			&S^*A-AS=T'(S^*-S)+K_1\\
			\implies & S^* A-AS=S^*T'-T'S+K_2 \quad (K_2 \text{ compact})\\
			\implies & S^*(A-T')-(A-T')S=K_2
		\end{align*}
		which implies $H'=A-T'$ is an essentially Hankel operator. Hence $A=H'+T' \in \clt_{ess}+ \clh_{ess}$. This completes the proof.
\end{proof}

As a by-product of this theorem, we provide a characterization for the essential commutant of the operator $S+S^*$ on $H^2_\cle(\D)$ for finite-dimensional Hilbert space $\cle$. 
	
\begin{cor}
Let $X \in \clb(H^2_{\cle}(\D))$. Then $X$ commutes with $S+S^*$ modulo compact operator if and only if $X \in \clt_{ess}+ \clh_{ess}$.
\end{cor}
	
\begin{proof}
Let $X =T+H \in \clt_{ess}+ \clh_{ess}$, where $T \in \clt_{ess}$ and 
$H \in  \clh_{ess}$.
		Consider
		\[
		(S+S^*)X-X(S+S^*)= (ST-TS)+(S^*T-TS^*)+(SH-HS^*)+(S^*H-HS).
		\]
Each summand of the right hand side is compact and hence the result holds. 

Conversely, let $(S+S^*)X-X(S+S^*)=K$(compact). Then
		\begin{align*}
			&S^*(S+S^*)XS-S^*X(S+S^*)S=S^*KS.
		\end{align*}
It follows that $XS+S^{*2}XS-S^*XS^2-S^*X= K_1$, where $K_1=S^*KS$ (compact). Thus by Theorem \ref{thm:maintheorem}, $X \in \clt_{ess}+ \clh_{ess}$.
\end{proof}
	
\vspace{0.5cm}	

\NI\textit{Acknowledgements:} 
The third author's research work is supported in part by the Core Research Grant (CRG/2022/006891), ANRF, Government of India.

\vspace{0.2cm}	
\noindent\textit{Conflict of interest.} The authors declare that they have no conflicts of interest. 

\vspace{0.2cm}	
\noindent\textit{Data availability.} This work is purely theoretical, and no datasets are used here.

\end{document}